\documentclass{amsproc}

\subjclass[2000]{Primary 46J10; Secondary 46H20}
\keywords{uniform algebras, amenability, regularity}

\usepackage{amsfonts}
\usepackage{amsthm}
\usepackage{amsmath}

\newcommand{\C}{\mathbb C}

\newcommand{\N}{\mathbb N}

\newcommand{\eps}{\varepsilon}
\newcommand{\norm}[1]{\Vert #1\Vert}
\newcommand{\abs}[1]{\vert #1 \vert}
\newcommand{\seq}[2]{(#1_{#2})_{#2\in\N}}

\newtheorem{theorem}{Theorem}[section]

\newtheorem{lemma}[theorem]{Lemma}

\theoremstyle{definition}

\newtheorem{Lemma}[theorem]{Lemma}

\theoremstyle{definition}

\newtheorem{prop}[theorem]{Proposition}
\theoremstyle{definition}
\newtheorem{dfn}[theorem]{Definition}
\newtheorem{question}[theorem]{Question}

\def\unorm#1{\|#1\|_{\infty}}

\title{Regularity and amenability conditions for uniform algebras}
\begin{document}
\author{J.F. Feinstein}
\address{School of Mathematical Sciences\\
University of Nottingham\\
Nottingham, NG7~2RD\\
UK.}

\email{Joel.Feinstein@nottingham.ac.uk}

\author{M.J. Heath}
\address{School of Mathematical Sciences\\
University of Nottingham\\
Nottingham, NG7~2RD
\\UK.}

\email{matthew.heath@maths.nottingham.ac.uk}
\thanks{The second author was supported by a grant from the EPSRC}
\maketitle
\begin{abstract}
 We give a survey of the known connections between regularity conditions and amenability conditions in the setting of uniform algebras.
For a uniform algebra $A$ we consider the set, $A_{lc}$, of functions in $A$ which are locally constant
on a (varying) dense open subset of the character space of $A$. We show that, for a separable uniform algebra $A$,
if $A$ has bounded relative units at every point of a dense subset of the character space of $A$, then $A_{lc}$ is dense in $A$. We construct a separable, essential,  regular uniform algebra $A$ on its character space $X$ such that every point of $X$ is a peak point for $A$, $A$ has bounded relative units at every point of a dense open subset of $X$ and yet $A$ is not weakly amenable.
In particular, this shows that a continuous derivation from a separable, essential uniform algebra $A$ to its dual need not annihilate
$A_{lc}$.
\end{abstract}

\section{Introduction}
\noindent Amenability conditions for Banach algebras and regularity conditions for Banach function algebras have been extensively studied. We refer the reader to \cite{Dales} as a reference for a detailed introduction to the majority of the conditions we discuss, their history, examples and applications.

Algebra amenability conditions originally arose out of Kamowitz's cohomology theory for Banach algebras (see
\cite{Kam}). The connection with the amenability of topological groups is given by a famous theorem of B.~E.~Johnson: for a locally compact topological group $G$,
the Banach algebra $L^1(G)$ is amenable if and only if $G$ is an amenable topological group (\cite{Johnson}).
The study of cohomology of Banach algebras continues to be a highly active area, and it is impossible to do it justice in a short survey. The reader may wish to consult \cite[Section 2.8]{Dales} (especially the historical summary beginning on page 304).

Regularity conditions have
important applications in areas such as automatic continuity theory (\cite[Chapter 5]{Dales}), the theory of Wedderburn decompositions 
(\cite{BadeDales}) and the decomposability of multiplication operators (\cite{Neumann}).
For more details of these applications, we refer the reader to \cite[Chapter 4]{Dales}.

\par In this note we give a brief survey, primarily in the setting of uniform algebras, of amenability and regularity conditions and the connections between them.
We also prove some new results concerning functions which are locally constant on a dense open subset of the
character space of a uniform algebra.

\section{Notation and terminology}
\noindent We shall assume that the reader has some familiarity with uniform algebras and with Banach function algebras (commutative, semisimple
Banach algebras). The reader may find the relevant definitions and the basic theory of Banach function algebras, along with numerous
examples, in \cite[Chapter 4]{Dales}.
Throughout, when we use the terms ``Banach function algebra $A$ on $X$'' and ``uniform algebra $A$ on $X$'' it will be assumed that
$X$ is the character space of $A$ and that $A$ is unital. In particular $X$ will always be a compact, Hausdorff space. We denote by  $C(X)$ the  uniform algebra of all continuous, complex-valued functions on $X$, and,
for $f \in C(X)$, we denote the uniform norm of $f$ by $\unorm{f}$.
For a non-empty, compact plane set $X$ we define $R_0(X)$
to be the set of restrictions to $X$ of rational functions with poles off $X$, and $R(X)$ to be the uniform closure of $R_0(X)$ in
$C(X)$.
Whenever $R(X)$ is referred to, $X$ will be assumed to be a non-empty, compact plane set.

We say that a uniform algebra is \emph{trivial} if
it is equal to $C(X)$; otherwise it is \emph{non-trivial}.

We shall frequently refer to the following ideals.
\begin{dfn}
 Let $A$ be a Banach function algebra on $X$ and let $x\in X$. We define ideals $J_x$ and $M_x$ by:
\begin{eqnarray*}
 M_x&:=&\{f\in A:f(x)=0\}\,;\\
J_x&:=&\left\{f\in A: f^{-1}(\{0\}) \textrm{ is a neighbourhood of } x\right\}.
\end{eqnarray*}
\end{dfn}
\section{Survey of conditions and known relationships}
\subsection{Amenability conditions}

\noindent Let $A$ be a Banach algebra and let $E$ be a Banach $A$-bimodule. A \emph{derivation }$D:A\rightarrow E$ is a linear map such that
$$D(ab)=D(a)\cdot b+a\cdot D(b)\quad(a,b\in A).$$
A derivation is: \emph{inner} if it is of the form $D(a)=a\cdot e -e\cdot a$ for some fixed $e\in E$;  \emph{approximately inner}
if it is in the strong operator topology closure of the inner derivations; \emph{pointwise inner} if for each $a\in A$ there is $e\in E$ with $D(a)=a\cdot e -e\cdot a$. Note
that, for a Banach algebra $A$ and a commutative Banach $A$-bimodule $E$, pointwise inner derivations (and therefore inner
derivations) are zero.

 The \emph{dual Banach $A$-bimodule to} $E$ is the Banach space dual $E'$ of $E$, together with the module actions given by
\begin{eqnarray*}
(a\cdot\phi)(e)&:=&\phi(e\cdot a)\,,\\
(\phi\cdot a)(e)&:=& \phi(a\cdot e),
\end{eqnarray*}
for $a\in A$, $e\in E$ and $\phi\in E'$.

\begin{dfn}
A Banach algebra $A$ is said to be: \emph{amenable} if all continuous derivations from $A$ to dual Banach $A$-bimodules are
inner;
\emph{pointwise amenable} if all continuous derivations from $A$ to dual Banach $A$-bimodules are pointwise inner;
\emph{approximately amenable} if all continuous derivations from $A$ to dual Banach $A$-bimodules are approximately inner;
\emph{weakly amenable} if all continuous derivations from $A$ to $A'$ are inner.
\end{dfn}
The term amenability (for Banach algebras) was introduced by B.~E.~Johnson in \cite{Johnson}.
Weak amenability was introduced (for commutative Banach algebras) by Bade \emph{et al.} in \cite{BCD}. Approximate
amenability was introduced by Ghahramani and Loy in \cite{GhahLoy} and pointwise amenability by Dales and Ghahramani in
\cite{DG}.  In \cite{GLZ} Ghahramani, Loy and Zhang showed that if a Banach algebra $A$ is approximately amenable, then all continuous derivations into \emph{any} Banach $A$-bimodule are approximately inner (that is, $A$ is \emph{approximately contractible} in the terminology of \cite{GhahLoy}).

The following was proved in \cite{BCD} and can also be found as \cite[2.8.63 (iii)]{Dales}
\begin{prop}
Let $A$ be a weakly amenable, commutative Banach algebra. Then, for every commutative Banach $A$-bimodule $E$,
there are no non-zero bounded derivations from $A$ into $E$.
\end{prop}

Let $\psi$ be a character on $A$. A point derivation at $\psi$ is a linear functional $d$ such that
$$d(ab)=\psi(a)d(b)+\psi(b)d(a),\quad a,b\in A.$$

In the case where $A$ is a uniform algebra we have the following relationships between amenability conditions:

$A=C(X) \Leftrightarrow A$ is amenable $\Leftrightarrow A$ is pointwise amenable $\Rightarrow A$ is approximately amenable $\Rightarrow A$ is weakly amenable $\Rightarrow A$ has
no non-zero bounded point derivations.

The first equivalence is \cite[Theorem 5.6.2]{Dales} and is due to She\u{\i}nberg. The second equivalence and the implication which follows it are due to Dales and Ghahramani in \cite{DG}.  The remaining stated implications are elementary.
That the final implication can not be reversed was shown, by the first author, in \cite{FeinsteinMorris}, by constructing a compact plane
set $X$ such that the uniform algebra $R(X)$  has no non-zero, bounded point derivations but is not weakly amenable. That there is an
$X$ such that $R(X)$ is non-trivial and has no non-zero bounded point derivations had previously been proved by Wermer in \cite{Wermer}.
The remaining reverse implications are open for uniform algebras. Moreover, it is open whether or not every weakly amenable uniform algebra must be trivial.

A separate set of conditions concerns the non-existence of (not necessarily bounded) point derivations. It is standard (see, for example,
\cite[p.64]{Browder}, \cite[p.267]{Dales}) that, for a Banach function algebra $A$ on $X$, and $x\in X$, $A$ has no non-zero point
derivations at $x$ if and only if $M_x^2=M_x$; $A$ has no non-zero {\em bounded} point derivations at $x$
if and only if
$\overline {M_x^2}=M_x$. We say that $Y\subseteq X$ is a \emph{peak set for} $A$ if there is $f\in A$ with $f(Y)=\{1\}$ and
$\abs{f(x)}<1$, for all $x\in X\setminus Y$; $x\in X$ is a \emph{peak point for $A$} if $\{x\}$ is a peak set for $A$; $x\in X$ is a
\emph{$p$-point for $A$} if $\{x\}$ is an intersection of peak sets for $A$.
Clearly if $x$ is a peak point for $A$ it is a $p$-point for $A$. If $X$ is metrizable, then the converse also holds: peak points and $p$-points coincide. Note that (as we are dealing with unital uniform algebras), $X$ is metrizable if and only if $A$ is separable.
By \cite[Theorem 4.3.5]{Dales}, for $x \in X$, $M_x$ has a bounded approximate identity
if and only if $x$ is a $p$-point for $A$; this, in turn, implies that $M_x^2=M_x$
(i.e. there are no non-zero point derivations at $x$), by \cite[Corollary 2.9.30(ii)]{Dales} (part of the Cohen factorisation theorem). 

The algebras $R(X)$ give a good source of examples in this area. In many cases, stronger examples can then be constructed by applying a suitable system of Cole root extensions to such an algebra.
For more details concerning such systems of extensions see, for example, \cite{Cole}, \cite{FeinsteinStronglyRegular}, and \cite{Dawson}.
The first author (\cite{FeinsteinMorris}) used these methods to prove the following result.

\begin{prop} \label{Jpp} There exists a separable uniform algebra $A$ such that every point of the
character space of $A$ is a peak point for $A$, but $A$ is not weakly amenable.
\end{prop}

As mentioned above, if $A$ is a separable uniform algebra on $X$, a point
$x\in X$ is a $p$-point for $A$ if and only if it is a peak point for $A$. For $A=R(X)$ these conditions are, further,  equivalent to the condition that $M_x^2 = M_x$ (see \cite[Corollary 3.3.11]{Browder}).
For general uniform algebras, it is not true that $M_x^2=M_x$ implies that $x$ is a $p$-point (see, for example, \cite{Cole,Ouzomgi}). It is open whether or not this implication holds for all separable uniform algebras.

Bishop's theorem (\cite[Theorem 3.3.3]{Browder}) states that
$R(X)=C(X)$ if and only if almost every (with respect to area) point of $X$ is a peak point for $R(X)$. Hence the previously mentioned 
example of Wermer \cite{Wermer} has many (unbounded) point derivations. Thus, even for $R(X)$, having no non-zero bounded point derivations does not
imply the non-existence of non-zero point derivations.

We note that in more general settings the situation is somewhat different. For the relationships between amenability conditions for
general Banach algebras see \cite[Chapter 2.8]{Dales}, \cite{GhahLoy} and \cite{DG}. It is not known if a pointwise amenable, commutative Banach algebra must be amenable. It is shown in \cite{BCD} that a Banach function algebra can
be weakly amenable and not amenable. Since, for a discrete abelian topological group $G$, $L^1(G)$ may be regarded (via the Gelfand transform) as
a Banach function algebra (\cite[4.5.4]{Dales}), \cite{Johnson} provides many examples of Banach function algebras which are amenable
and not equal to $C(X)$. Also, for Banach function algebras, there are simpler examples to show that having no non-zero point derivations
does not imply weak amenability; the well known algebras $AC[0,1]$ and $BVC[0,1]$ will suffice (\cite[Theorems 4.4.35 and 5.6.8]{Dales}).
Examples of Banach function algebras which are approximately amenable but not amenable may be found in
\cite{DLD}. One such example is the semigroup algebra $\ell^1(S)$, where $S$ is
the set of non-negative integers with the semigroup multiplication $m\cdot n = \max\{m,n\}$ $(m, n \in S)$.

\subsection{Regularity conditions}
\subsubsection{Regularity properties at a point}

\begin{dfn}
Let $A$ be a Banach function algebra on $X$. A point $x\in X$ is a \emph{point of continuity for}  $A$ if there is no point $y \in
X\setminus \{x\}$, such that $M_x\supseteq J_y$. We say that $A$ is \emph{strongly regular at} $x$ if $\overline{J_x}=M_x$;  $A$ satisfies
\emph{Ditkin's condition at} $x$ if, for all $f\in M_x$ and all $\eps > 0$, there is $g\in J_x$ with $\norm{gf-f}<\eps$; for $C \geq 1$, $A$
has \emph{bounded relative units at} $x$ \emph{with bound} $C$ if,
for all compact $E\subseteq X\setminus\{x\}$, there is $f\in J_x$ with
$\norm{f}\le C$, such that $f(E)\subseteq \{1\}.$
We say that $A$ \emph{has bounded relative units at} $x$ if there exists $C \geq 1$ such that $A$ has bounded relative units at $x$ with bound $C$.
\end{dfn}
We may  consider the following conditions which a Banach function algebra $A$ on $X$ may satisfy at a point $x\in X$:
\begin{itemize}
 \item [(i)] for all $C>1$, $A$ has bounded relative units at $x$ with bound $C$;
\item[(ii)] $A$ has bounded relative units at $x$;
\item[(iii)] $A$ is strongly regular at $x$ and $M_x$ has a bounded approximate identity;
\item[(iv)] $A$ satisfies Ditkin's condition at $x$;
\item[(v)] $A$ is strongly regular at $x$;
\item[(vi)] $x$ is a point of continuity for $A$.
\end{itemize}
In the case where $A$ is a uniform algebra, the following relationships exist between these conditions:
$$\textrm{(i)} \Leftrightarrow\textrm{(ii)} \Rightarrow \textrm{(iii)} \Rightarrow \textrm{(iv)}\Rightarrow \textrm{(v)}\Rightarrow
\textrm{(vi)}.$$

\subsubsection{Global regularity properties}
\begin{dfn}
 Let $A$ be a Banach function algebra on its character space $X$. We say that: $A$ is \emph{regular} if, for all $x\in X$ and all
compact sets $E\subseteq X\setminus\{x\}$ there exists $f\in M_x$ with $f(E)\subseteq 1$; $A$ is \emph{strongly regular} if it is
strongly regular at each $x\in X$; $A$ is a \emph{Ditkin algebra} if it satisfies Ditkin's condition at each $x\in X$; $A$ is a
\emph{strong Ditkin algebra} if it is strongly regular and every maximal ideal has a bounded approximate identity; $A$ has \emph{bounded
relative units} if it has bounded relative units at each $x\in X$. For $C\geq 1$, $A$ has \emph{bounded relative units with global bound $C$} if it has bounded relative units with bound $C$ at each $x\in X$.
\end{dfn}
It is standard (see, for example, \cite{FeinsteinTrivJen,FS}) that $A$ is regular if and only every point of $X$ is a point of continuity for $A$.

We may consider the following conditions which a Banach function algebra may satisfy:
\begin{itemize}
 \item [(a)]for all $C>1$, $A$ has bounded relative units with global bound $C$;
\item[(b)]$A$ has bounded relative units;
\item[(c)]$A$ is a strong Ditkin algebra;
\item[(d)]$A$ is a Ditkin algebra;
\item[(e)]$A$ is strongly regular;
\item[(f)]$A$ is regular.
\end{itemize}
For uniform algebras the following relationships exist between these conditions:
$$\textrm{(a)} \Leftrightarrow \textrm{(b)} \Leftrightarrow \textrm{(c)}\Rightarrow \textrm{(d)}\Rightarrow
\textrm{(e)}\Rightarrow \textrm{(f)}.$$
The uniform algebra $C(X)$ satisfies all of these conditions. An example of a non-trivial uniform algebra with bounded relative units is given in
\cite{FeinsteinStronglyRegular}. Also in \cite{FeinsteinStronglyRegular} is a strongly regular uniform
algebra without bounded relative units (so (e) does not imply (c)). For uniform algebras, it is not known whether (e) implies (d), or whether (d) implies (c).
That (f) does not imply (e) is discussed in the following section.

For both the global regularity conditions and those at a point, the situation is rather different for general Banach function algebras. 
This, more general, case is surveyed in \cite{FeinsteinRegCon} and \cite[Section 4.1]{Dales} (see also
\cite{FeinsteinStrDit}).

\subsection{Amenability and regularity conditions together}
Let $A$ be a strong Ditkin algebra on $X$. Then every maximal ideal
$M_x$ has a bounded approximate identity, and so there are no non-zero point derivations at all on $A$.
For the special case of $R(X)$ we have that, if $R(X)$ is a strong Ditkin
algebra, then every $x\in X$ is a peak point and hence, by Bishop's theorem, $R(X)=C(X)$.
In \cite{O'Farrell}, O'Farrell constructs a compact plane set $X$ such that $R(X)$ is regular but has a non-zero
bounded point derivation.
If $A$ is regular and is strongly regular at $x$ then an immediate consequence of \cite[4.1.20(iv)]{Dales} is that
$\overline{J_x}=\overline{M_x^2}=M_x$. Hence, $A$ has no non-zero bounded point derivations at $x$. In particular, O'Farrell's example 
is not strongly regular, and thus (f) does not imply (e) in the previous subsection. A separable, regular uniform algebra which has no non-zero bounded point
derivations, but which is not strongly regular, can be constructed by making a small modification to the construction of the uniform algebra found in the
remark on page 300 of \cite{FeinsteinStronglyRegular}. All that is needed is to ensure that, in the construction on page 299 of \cite{FeinsteinStronglyRegular}, each family ${\mathcal F}_\alpha$ is dense in $I(E_\alpha)$.
It is not known whether it is possible for $R(X)$ to have this combination of properties.

In \cite{Me} the second author (based on the work of the first in \cite{FeinsteinMorris}) proved the following result. (The reader may refer to \cite[2-8]{Browder} for information on the essential set and
essential uniform algebras.)

\begin{prop} \label{thm}Let $Q=\{x+iy\in\C:x,y\in[0,1]\}$. For each $C>0$ there is a compact set $X\subseteq Q$  such that
$\partial Q$ is a subset of $X$, $X\setminus \partial Q$ is dense in $X$, $R(X)$ is regular and has no non-zero, bounded point
derivations and, for all $f, g$ in $R_0(X),$
\begin{displaymath} \bigg\vert\int_{\partial Q}f^\prime(z)g(z)\textrm{d}z\bigg\vert\le C \unorm{f}\unorm{g}\,.
\end{displaymath}
In particular, $R(X)$ is not weakly amenable. Furthermore, the following are true:
\begin{itemize}
 \item $\partial Q$ is contained in the essential set, $E$, of $R(X)$;
\item for $f \in R(X)$, $D(f)$ depends only on $f|_{\partial Q}$;
\item $R(E)=R(X)|_E$;
\item  $E\setminus\partial Q$ is dense in $E$.
\end{itemize}
\end{prop}
These last four itemised conditions are not stated in \cite{Me} but can be seen from the construction; the first three allow us to
assume that our example is essential. Furthermore, \cite{Me} also contains an example of a separable uniform algebra $B$ on
$Y$, such that $B$ is regular, not weakly amenable and such that every maximal ideal has a bounded approximate identity. We note that an
example of a Banach function
algebra with these properties is much easier to produce; $AC[0,1]$ and $BVC[0,1]$ will again suffice and are, in addition, strong Ditkin
algebras (\cite[Theorems 4.435 and 5.6.8]{Dales}).

\section{Locally constant functions and derivations}
\noindent The set of continuous real-valued functions on a compact space, $X$, which are locally constant on a (varying) dense open subset of $X$ is
discussed  by Bernard and Sidney (\cite{BernardSidney}) and by Sidney (\cite {Sidney}). We shall consider this property in the complex 
case. For a Banach function algebra, $A$, on its character space, $X$, we define the  following set.
\begin{eqnarray*}
 A_{lc}&=&\{f\in A: \textrm { there exists a dense open set }U\subseteq X \\
&&\textrm{ with $f$ locally constant on } U\}.
\end{eqnarray*}

The following results show how the denseness of $A_{lc}$ in $A$ is related to having bounded relative units.
\begin{lemma}\label{loc con1}
Let $A$ be a uniform algebra on $X$ and let $x \in X$. Suppose that $A$ has bounded relative units at $x$.
Let
$U$ be a neighbourhood of $x$, let $f\in A$ and let $\eps>0$. Then there is a function $g\in A$ which is constant on a neighbourhood of $x$, agrees with $f$ on
$X\setminus U$ and has $\unorm{g-f}<\eps$.
\end{lemma}
\begin{proof}
Let $k \geq 1$ be such that $A$ has bounded relative units at $x$ with bound $k$. Let $V$ be a neighbourhood of $x$ such that $V\subseteq U$ and, for each
$y\in V$,
\[\abs{f(y)-f(x)}<\frac{\eps}{k+1}\,.\]
Choose $h\in J_x$ with $\unorm{h} \leq k$ and such that $h(X\setminus V)\subseteq \{1\}$, and set \[g=(f-f(x))h+f(x)\,.\]
Then $g$ has the required properties.
\end{proof}

\begin{theorem}\label{loc con}
 Let $A$ be a uniform algebra on a compact metric space $X$, and suppose that $Y$ is a dense subset of $X$ such that $A$ has bounded
relative units at $x$ for each $x\in Y$. Then $A_{lc}$ is dense in $A$.
\end{theorem}
\begin{proof}
Without loss of generality, $Y$ is at most countable. In the case where $Y$ is finite the result is trivial so we assume that $Y$ is
countably infinite. We enumerate $Y$ as a sequence $\seq{x}{k}$.

Let $f\in A$ and $\eps>0$. We set $g_0=f$ and $U_0=\emptyset$ and
define, inductively, sequences $\seq{g}{k}\subseteq A$ and $\seq{U}{k}\subseteq X$ as follows. Fix $k\in\N$. If
$\overline {U_{k-1}}=X$ then define $g_k=g_{k-1}$. Otherwise, set
$$n_k=\min\left\{j\in\N: x_j\not\in \overline{U_{k-1}}\right\}.$$
By Lemma \ref{loc con1} we may choose a $g_k\in A$ which is constant on a neighbourhood of $x_{n_k}$,
equal to $g_{k-1}$ on $U_{k-1}$ and with $\unorm{g_{k}-g_{k-1}}<2^{-n}\eps$. In either case, we set
$$U_k:=\{x\in X:g_k \textrm{ is constant on a neighbourhood of } x\}.$$
The inductive choice may now proceed.

Clearly, we have $U_1 \subseteq U_2 \subseteq U_3 \subseteq \cdots$, $\bigcup_{k=1}^\infty U_k$ is dense in $X$, and $\seq{g}{k}$ is a Cauchy sequence in $A$.
Set
$g:=\lim_{k\rightarrow\infty} g_k.$
Then $\unorm{g-f}<\eps$.
We show that $g$ is in $A_{lc}$ by showing that it is locally constant on
$\bigcup_{k=1}^\infty U_k$.
Let $x\in U_k$, for some $k\in\N$. Choose an open neighbourhood $V$ of $x$ with $V\subseteq U_k$ and such that
$g_k$ is constant on $V$. Since, for each  $n\ge k$, we have $g_n\vert_{U_k}=g_k\vert_{U_k}$, we see that $g$ is constant on $V$. Thus $g$ is locally constant on $\bigcup_{k=1}^\infty U_k$, and hence
$g \in A_{lc}$.
The result follows.
\end{proof}

Based on the intuitive notion that bounded derivations from a uniform algebra are ``like differentiation'' a  na\"{\i}ve
conjecture concerning $A_{lc}$ is the following: each bounded derivation $D$ from a uniform algebra $A$ on $X$ into a commutative Banach $A$-bimodule annihilates $A_{lc}$.
This is easily seen to be false for Banach function algebras: any Banach function algebra
with character space equal to the one point compactification of $\N$ and which has
a bounded point derivation at infinity will do to give a counterexample. Furthermore, we can see that the conjecture is also false for uniform algebras by considering the well known ``tomato can
algebra'' (see \cite[2-8]{Browder}), consisting of the
continuous functions on a solid cylinder which are analytic on one face. However this is rather an
artificial construction, and is far from essential. Hence the question arises as to whether the conjecture holds for essential uniform algebras. We shall show that the answer to this question is still negative.

We begin by proving some preliminary results. The lemma below is a ``point-by-point'' version
of \cite[Lemma 3.4]{FeinsteinStronglyRegular}.

\begin{Lemma}\label{bru}
 Let $A$ be a uniform algebra on $X$ and $x\in X$. Suppose that, for each compact subset $E$ of $X\setminus\{x\}$,
there exists an open neighbourhood, $U$, of $x$, and $f\in A$ such that
\begin{itemize}
 \item [(i)]$f(U)= \{1\}$,
\item[(ii)]$f(E)\subseteq\{0\}$,
\item[(iii)] For each $k\in \N$ there is a $g\in A$ with $g^{2^k}=f$.
\end{itemize}
Then $A$ has bounded relative units at $x$.
\end{Lemma}
\begin{proof}
Let $E$ be a compact subset of $X\setminus\{x\}$.
Choose an open neighbourhood $U$ of $x$ and $f\in A$ satisfying the above conditions (i) to (iii).
Let $k \in \N$ be such that $\unorm{f}^{2^{-k}}\le 2$, and choose $g\in A$ with $g^{2^k}=f$. Without loss of generality, we may assume that
$g(x)=1$.
Then $\unorm g \le 2$ and $g(E)\subseteq \{0\}$.
We have $f(y)=g^{2^k}(y)=1$, for $y\in U$, and so
$g$ takes only finitely many different values on $U$. Thus, since $g$ is continuous, $g$ must be constantly
equal to $1$ on
some neighbourhood of $x$.
Set $h=1-g$. Then $h \in J_x$, $h(E) \subseteq \{1\}$ and $\unorm{h} \leq 3$. Thus $A$ has bounded relative units at $x$.
\end{proof}
The following lemma is \cite[Lemma 3.5]{FeinsteinStronglyRegular}
\begin{prop}\label{regular}
 Let $A$ be a regular uniform algebra on a compact metric space, $X$. Then there exists a countable set, $\mathcal F\subseteq A$, such 
that, for each closed subset $E$ of $X$ and each $x\in X\setminus E$, there exists an open set, $U$, containing $x$, and $f\in\mathcal F$ with $f(U)=\{1\}$ and
$f(E)\subseteq\{0\}$.
\end{prop}

We are now ready to construct our example. The construction is similar to those of \cite[Theorem 3.1]{FeinsteinMorris} and
\cite[Theorem 3.2]{Me}, and the proof will be a sketch, leaving out some of the details, many of which may be found in those papers (primarily in
\cite{FeinsteinMorris}). For a discussion of the notation relating to  Cole extensions, and of their basic properties, see \cite{Cole}, \cite{Dawson}
and \cite{FeinsteinStronglyRegular}.
\begin{theorem}\label{bru-notwa}
There is a separable, essential, regular uniform algebra $A$ on $X$ and a dense open subset $V$ of $X$ such that every point of $X$ is a peak point for $A$ and $A$ has bounded relative units at every point of $V$, but $A$ is not weakly amenable.
\end{theorem}
\begin{proof} Let $X_0$ be the compact plane set called $E$ in Proposition \ref{thm} (so $X_0$ is
the essential set of the uniform algebra considered in that proposition). Set $A_0=R(X_0)$, so that $A_0$ is essential. Let $\omega$ be the first
infinite ordinal. We construct, inductively, a countable system of root extensions
\[\left((X_\alpha;\Pi_{\alpha, \beta}),(A_\alpha)\right), \quad(0\le\alpha\le\beta\le\omega)\,,\] so $X_\omega$ is the inverse limit
of $\left(X_\alpha;\Pi_{\alpha, \beta}\right)~(0\le\alpha\le\beta <\omega)$, and $A_\omega$ is the direct limit of
$\left(A_\alpha;\Pi^*_{\alpha, \beta}\right)~(0\le\alpha\le\beta <\omega)$. We shall ensure that $A_\omega$ will be a
uniform algebra on $X_\omega$ with the required properties.

At stage $\alpha<\omega$, let $F_\alpha:= \Pi_{0,\alpha}^{-1}( \partial Q)$. Proposition \ref{regular} lets us define
$A_{\alpha+1}$ by attaching square roots to a countable subset $\mathcal F_\alpha$ of $A_\alpha$ with the following properties:
\begin{enumerate}
 \item $\mathcal F_\alpha$ is a dense subset of the ideal of functions in $A_\alpha$ which vanish on
$F_\alpha$.
\item For each $x\in X_\alpha\setminus F_\alpha$ and each compact subset $E$ of $X_\alpha\setminus \{x\}$ there exists an open subset $U$ of $X_{\alpha}$ containing $x$, and $f \in \mathcal F_\alpha$ with $f(U)=\{1\}$ and
$f(E)\subseteq\{0\}$.
\end{enumerate}
We may also ensure at each stage that every element of $\Pi_{\alpha, \alpha+1}^*\mathcal F_\alpha$ has a square root in $\mathcal F_{\alpha+1}$.

Set $A=A_\omega$ and $X=X_\omega$ and define $F_\omega := \Pi_{0,\omega}^{-1}( \partial Q)$.
As in the proofs of \cite[3.1]{FeinsteinMorris} and \cite[3.2]{Me}, condition (1), above, ensures
that $A$ is a regular uniform algebra with character space $X$, $X$ is metrizable, $A$ is not weakly amenable, and every $x\in X$ is a 
peak point for $A$. Let $V:=X\setminus F_\omega$. It is easily seen that $V$ is a dense open subset of $X$. We shall show that,
for each $x\in V$, $A$ has bounded relative units at $x$ by showing that, at each such $x$, the conditions of Lemma \ref{bru} hold. Set 
$$\mathcal F=\bigcup \left\{\Pi_{\alpha, \omega}^*(\mathcal F_a):\alpha<\omega\right\}.$$
Each element of $\mathcal F$ has a square root in $\mathcal F$ and, therefore, also has a $2^k$th root in $\mathcal F$, for all
$k\in \N$. Let $x\in V$ and let $E$ be a compact subset of $X\setminus \{x\}$. Since $X$ is the inverse limit of the inverse system
$(X_\alpha;\Pi_{\alpha, \beta}),\,(\alpha\le\beta<\omega)$, and $x\in X\setminus E$, there is $\alpha<\omega$ with
$\Pi_{\alpha, \omega}(x)\in X_\alpha\setminus\Pi_{\alpha, \omega}(E)$. There exists an open subset, $W$, of $X_\alpha$, containing
$\Pi_{\alpha, \omega}(x)$ and $g\in\mathcal F_\alpha$ with $g(\Pi_{\alpha, \omega}(E))\subseteq\{0\}$ and $g(W)=\{1\}$. Set
$f=\Pi_{\alpha, \omega}^*g\in\mathcal F$ and set $U=\Pi_{\alpha, \omega}^{-1}(W)$. We have that $f$ and $U$ satisfy the conditions of
Lemma \ref{bru}, as required.

Finally, it may be shown that, since $A_0$ is essential, so is $A$ and the result follows.
\end{proof}

Note that, by Lemma \ref{loc con}, $A_{lc}$ is dense in $A$. Since $A$ is not weakly amenable, there is a bounded derivation from $A$ to $A'$
which does not annihilate $A_{lc}$.

\section{Open Questions}

\subsection{Questions for uniform algebras}
The following questions relate to arbitrary uniform algebras.

\begin{question}
Is there a non-trivial, weakly amenable uniform algebra?
\end{question}
\begin{question}
Is there a non-trivial, approximately amenable uniform algebra?
\end{question}
\begin{question}
Is every weakly amenable uniform algebra approximately amenable?
\end{question}
\begin{question}
Is there a uniform algebra which has bounded relative units but which is not weakly amenable?
\end{question}
\begin{question}
Is there a uniform algebra which is strongly regular but which is not weakly amenable?
\end{question}
\begin{question}
More generally, what implications are there between the various regularity and amenability conditions we have discussed for uniform algebras?
\end{question}

\subsection{Questions for $R(X)$}
The following questions relate to the special case of $R(X)$.

\begin{question}
Is there a compact plane set such that $R(X)$ is non-trivial and strongly regular?
\end{question}

\begin{question}
Is there a compact plane set $X\subseteq\C$ such that $R(X)$ is regular and has no non-zero bounded point derivations, but is
not strongly regular?
In particular, is the uniform algebra $R(X)$ from Proposition \ref{thm} (and \cite{Me}) strongly regular?
\end{question}

\begin{question}
What are the answers to the questions for general uniform algebras in the case of $R(X)$?
\end{question}

\subsection{Questions on the interval}
The following questions relate to uniform algebras on the compact interval $[0,1]$.

\begin{question}
Is there a non-trivial uniform algebra with character space equal to $[0,1]$? (This is a famous open problem of Gelfand.)
\end{question}
\begin{question}
Is $C([0,1])$ the only regular uniform algebra on $[0,1]$? (For some partial results on this, see
\cite{FS-strong}).
\end{question}
\begin{question}
What are the answers to the questions for general uniform algebras in the case of uniform algebras on $[0,1]$?
\end{question}

\providecommand{\bysame}{\leavevmode\hbox to3em{\hrulefill}\thinspace}
\providecommand{\MR}{\relax\ifhmode\unskip\space\fi MR }
\providecommand{\MRhref}[2]{%
  \href{http://www.ams.org/mathscinet-getitem?mr=#1}{#2}
}
\providecommand{\href}[2]{#2}


\begin{thebibliography}{10}

\bibitem{BCD}
W.~G. Bade, P.~C. Curtis, Jr., and H.~G. Dales, \emph{Amenability and weak
  amenability for {B}eurling and {L}ipschitz algebras}, Proc. London Math. Soc.
  (3) \textbf{55} (1987), no.~2, 359--377.

\bibitem{BadeDales}
W.~G. Bade and H.~G. Dales, \emph{The {W}edderburn decomposability of some
  commutative {B}anach algebras}, J. Funct. Anal. \textbf{107} (1992), no.~1,
  105--121.

\bibitem{BernardSidney}
A.~Bernard and S.~J. Sidney, \emph{Some ultrabornological normed function
  spaces}, Arch. Math. (Basel) \textbf{69} (1997), no.~5, 409--417.


\bibitem{Browder}
A.~Browder, \emph{Introduction to function algebras}, W. A. Benjamin, Inc., New
  York-Amsterdam, 1969.

\bibitem{Cole}
B.~J. Cole, \emph{One point parts and the peak point conjecture}, Ph.D. thesis,
  Yale University, 1968.

\bibitem{Dales}
H.~G. Dales, \emph{Banach algebras and automatic continuity}, London
  Mathematical Society Monographs. New Series, vol.~24, The Clarendon Press
  Oxford University Press, New York, 2000, Oxford Science Publications.


\bibitem{DLD}
H.~G. Dales, A.T.M. Lau, and D.~Strauss, \emph{Banach algebras on semigroups
  and their compactifications}, submitted.

\bibitem{DG}
H.~G. Dales and F.~Ghahramani, \emph{Pointwise amenability}, in preparation.

\bibitem{Dawson}
T.~Dawson, \emph{A survey of algebraic extensions of commutative, unital normed
  algebras}, Function spaces (Edwardsville, IL, 2002), Contemp. Math., vol.
  328, Amer. Math. Soc., Providence, RI, 2003, pp.~157--170.

\bibitem{FeinsteinStronglyRegular}
J.~F. Feinstein, \emph{A nontrivial, strongly regular uniform algebra}, J.
  London Math. Soc. (2) \textbf{45} (1992), no.~2, 288--300.

\bibitem{FeinsteinRegCon}
\bysame, \emph{Regularity conditions for {B}anach function algebras}, Function
  spaces (Edwardsville, IL, 1994), Lecture Notes in Pure and Appl. Math., vol.
  172, Dekker, New York, 1995, pp.~117--122.

\bibitem{FeinsteinStrDit}
\bysame, \emph{Strong {D}itkin algebras without bounded relative units}, Int.
  J. Math. Math. Sci. \textbf{22} (1999), no.~2, 437--443.

\bibitem{FeinsteinTrivJen}
\bysame, \emph{Trivial {J}ensen measures without regularity}, Studia Math.
  \textbf{148} (2001), no.~1, 67--74.

\bibitem{FeinsteinMorris}
\bysame, \emph{A counterexample to a conjecture of {S}. {E}. {M}orris}, Proc.
  Amer. Math. Soc. \textbf{132} (2004), no.~8, 2389--2397 (electronic).

\bibitem{FS-strong}
J.~F. Feinstein and D.~W.~B. Somerset, \emph{Strong regularity for uniform
  algebras}, Function spaces (Edwardsville, IL, 1998), Contemp. Math., vol.
  232, Amer. Math. Soc., Providence, RI, 1999, pp.~139--149.

\bibitem{FS}
\bysame, \emph{Non-regularity for {B}anach function algebras}, Studia Math.
  \textbf{141} (2000), no.~1, 53--68.

\bibitem{GhahLoy}
F.~Ghahramani and R.~J. Loy, \emph{Generalized notions of amenability}, J.
  Funct. Anal. \textbf{208} (2004), no.~1, 229--260.

\bibitem{GLZ}
F.~Ghahramani, R.~J. Loy and Y.~Zhang, \emph{Generalized notions of amenability II}, preprint.

\bibitem{Me}
M.~J. Heath, \emph{A note on a construction of {J}. {F}.\ {F}einstein}, Studia
  Math. \textbf{169} (2005), no.~1, 63--70.

\bibitem{Johnson}
B.~E. Johnson, \emph{Cohomology in {B}anach algebras}, American Mathematical
  Society, Providence, R.I., 1972, Memoirs of the American Mathematical
  Society, No. 127.

\bibitem{Kam}
H.~Kamowitz, \emph{Cohomology groups of commutative {B}anach algebras}, Trans.
  Amer. Math. Soc. \textbf{102} (1962), 352--372.

\bibitem{Neumann}
M.~M. Neumann, \emph{Commutative {B}anach algebras and decomposable operators},
  Monatsh. Math. \textbf{113} (1992), no.~3, 227--243.

\bibitem{O'Farrell}
A.~G. O'Farrell, \emph{A regular uniform algebra with a continuous point
  derivation of infinite order}, Bull. London Math. Soc. \textbf{11} (1979),
  no.~1, 41--44.

\bibitem{Ouzomgi}
S. Ouzomgi, \emph{A commutative {B}anach algebra with factorization of elements but not of pairs} Proc. Amer. Math. Soc. \textbf{113}
(1991) no. ~2,  435--441.

\bibitem{Sidney}
S.~J. Sidney, \emph{Some very dense subspaces of {$C(X)$}}, Function spaces
  (Edwardsville, IL, 1994), Lecture Notes in Pure and Appl. Math., vol. 172,
  Dekker, New York, 1995, pp.~353--360.

\bibitem{Wermer}
J.~Wermer, \emph{Bounded point derivations on certain {B}anach algebras}, J.
  Functional Analysis \textbf{1} (1967), 28--36.

\end{thebibliography}
\end{document}